\documentclass[11pt]{article}

\title{Proof of the Arnold chord conjecture in three dimensions I}
\author{Michael Hutchings and Clifford Henry Taubes}
\date{}

\usepackage{amssymb}
\usepackage{latexsym}
\usepackage{amsmath}
\usepackage{amsthm}
\usepackage{amscd}

\newcommand{\mc}[1]{{\mathcal #1}}

\numberwithin{equation}{section}

\newtheorem{theorem}{Theorem}[section]
\newtheorem{proposition}[theorem]{Proposition}

\newtheorem{lemma}[theorem]{Lemma}
\newtheorem{lemma-definition}[theorem]{Lemma-Definition}

\theoremstyle{definition}
\newtheorem{definition}[theorem]{Definition}
\newtheorem{remark}[theorem]{Remark}

\newtheorem{example}[theorem]{Example}
\newtheorem{convention}[theorem]{Convention}

\newtheorem*{conv}{Convention}

\newcommand{\eqdef}{\;{:=}\;}

\renewcommand{\frak}{\mathfrak}

\newcommand{\R}{{\mathbb R}}

\newcommand{\Z}{{\mathbb Z}}

\newcommand{\op}{\operatorname}

\newcommand{\Spinc}{\op{Spin}^c}

\newcommand{\Ker}{\op{Ker}}

\newcommand{\bpm}{\begin{pmatrix}}
\newcommand{\epm}{\end{pmatrix}}

\renewcommand{\epsilon}{\varepsilon}

\begin{document}

\setcounter{tocdepth}{2}

\maketitle

\begin{abstract}
  This paper and its sequel prove that every Legendrian knot in a
  closed three-manifold with a contact form has a Reeb chord.
  The present paper deduces this result from
  another theorem, asserting that an exact symplectic cobordism
  between contact 3-manifolds induces a map on (filtered) embedded
  contact homology satisfying certain axioms.  The latter theorem will
  be proved in the sequel using Seiberg-Witten theory.
\end{abstract}

\section{Introduction}

\subsection{The chord conjecture}

Let $Y$ be a closed oriented $3$-manifold (all $3$-manifolds in this
paper will be assumed connected).  Recall that a {\em contact form\/}
on $Y$ is a $1$-form $\lambda$ on $Y$ with $\lambda\wedge d\lambda>0$
everywhere.  The contact form $\lambda$ determines a {\em contact
  structure\/}, namely the oriented 2-plane field $\xi=\Ker(\lambda)$.
It also determines the {\em Reeb vector field\/} $R$ characterized by
$d\lambda(R,\cdot)=0$ and $\lambda(R)=1$.  A {\em Legendrian knot\/}
in $(Y,\lambda)$ is a smooth knot $\mc{K}\subset Y$ such that
$T\mc{K}\subset\xi|_\mc{K}$.  A {\em Reeb chord\/} of $\mc{K}$ is a
trajectory of the Reeb vector field starting and ending on $\mc{K}$,
i.e.\ a path $\gamma:[0,T]\to Y$ for some $T>0$ such that
$\gamma'(t)=R(\gamma(t))$ and $\gamma(0),\gamma(T)\in \mc{K}$.  The
main result of this paper is:

\begin{theorem}
\label{thm:cc}
Let $Y$ be a closed oriented $3$-manifold with a contact form
$\lambda$.  Then every Legendrian knot in $(Y,\lambda)$ has a Reeb
chord.
\end{theorem}

This is a version of a conjecture of Arnold \cite{arnold}.  For the
3-sphere with any contact form inducing the standard contact
structure, and more generally for boundaries of subcritical Stein
manifolds in any odd dimension, this was proved by Mohnke
\cite{mohnke}.  This was also proved by Abbas \cite{abbas} for
Legendrian unknots in tight contact 3-manifolds satisfying certain
assumptions.

The proof of Theorem~\ref{thm:cc} given here uses the relationship
between embedded contact homology and Seiberg-Witten Floer cohomology.
We now recall the relevant parts of this correspondence.

\subsection{Embedded contact homology}
\label{sec:ech}

We begin by briefly reviewing the definition of embedded contact
homology.  For more details see \cite{icm} and the references therein.

Let $Y$ be a closed oriented 3-manifold with a contact form $\lambda$.
A {\em Reeb orbit\/} is a closed orbit of the Reeb vector field, i.e.\
a map $\gamma:\R/T\Z\to Y$ for some $T>0$ with
$\gamma'(t)=R(\gamma(t))$, modulo reparametrization. The linearized
Reeb flow along a Reeb orbit $\gamma$ defines an endomorphism
$P_\gamma$ of the 2-dimensional symplectic vector space
$(\xi_{\gamma(0)},d\lambda)$. A Reeb orbit $\gamma$ is {\em
  nondegenerate\/} if $P_\gamma$ does not have $1$ as an
eigenvalue. In this case either $P_\gamma$ has real eigenvalues, in
which case we say that $\gamma$ is {\em hyperbolic\/}, or $P_\gamma$
has eigenvalues on the unit circle, in which case $\gamma$ is called
{\em elliptic\/}. These notions do not depend on the parametrization
of $\gamma$. We say that the contact form $\lambda$ is {\em
  nondegenerate\/} if all its Reeb orbits are nondegenerate. A generic
contact form has this property.

Assume now that the contact form $\lambda$ on $Y$ is nondegenerate. An
{\em orbit set\/} is a finite set of pairs $\Theta=\{(\Theta_i,m_i)\}$
where the $\Theta_i$'s are distinct embedded Reeb orbits, and the
$m_i$'s are positive integers which one can think of as
``multiplicities''.  The homology class of the orbit set $\Theta$ is
defined by
\[
[\Theta] \eqdef \sum_i m_i[\Theta_i] \in H_1(Y).
\]
The orbit set $\Theta=\{(\Theta_i,m_i)\}$ is called {\em
  admissible\/} if $m_i=1$ whenever $\Theta_i$ is hyperbolic. An
admissible orbit set is also called an {\em
  ECH generator\/}.

If $\Gamma\in H_1(Y)$, then the {\em embedded
  contact homology\/} $ECH_*(Y,\lambda,\Gamma)$ is
the homology of a chain complex which is freely generated over $\Z/2$
by admissible orbit sets $\Theta$ with $[\Theta]=\Gamma$.

\begin{conv}
  Although ECH is ordinarily defined over $\Z$, with the signs specified
  in \cite[\S9]{obg2}, in this paper ECH is always defined with $\Z/2$
  coefficients, because this is sufficient for the applications here
  and will allow us to avoid orientation headaches.
\end{conv}

To define the chain complex differential
one chooses a generic almost complex structure $J$ on $\R\times Y$ of
the following type:

\begin{definition}
\label{def:SA}
An almost complex structure $J$ on $\R\times Y$ is {\em
  symplectization-admissible\/} if $J$ is $\R$-invariant,
$J(\partial_s)=R$ where $s$ denotes the $\R$ coordinate, and $J$ sends
$\xi$ to itself, rotating $\xi$ positively with respect to the
orientation on $\xi$ given by $d\lambda$.
\end{definition}

The reason for the terminology is that the noncompact symplectic
manifold $(\R\times Y,d(e^s\lambda))$ is called the {\em
  symplectization\/} of $(Y,\lambda)$.

Given a symplectization-admissible $J$, we now consider (not
necessarily embedded) $J$-holomorphic curves in $\R\times Y$ whose
domains are (not necessarily connected) punctured compact Riemann
surfaces.  A {\em positive end\/} of such a holomorphic curve at a
(not necessarily embedded) Reeb orbit $\gamma$ is an end which is
asymptotic to the cylinder $\R\times\gamma$ as the $\R$ coordinate
$s\to +\infty$.  A {\em negative end\/} is defined analogously with
$s\to-\infty$.

\begin{definition}
\label{def:Jhol}
Given a symplectization-admissible $J$, and given (not necessarily
admissible) orbit sets $\Theta=\{(\Theta_i,m_i)\}$ and
$\Theta'=\{(\Theta'_j,m'_j)\}$, a ``$J$-holomorphic curve from
$\Theta$ to $\Theta'$'' is a $J$-holomorphic curve in $\R\times Y$
as above with positive ends at
covers of $\Theta_i$ with total multiplicity $m_i$, negative ends at
covers of $\Theta'_j$ with total multiplicity $m'_j$, and no other
ends.

Such a holomorphic curve may be multiply covered, but we are only
interested in the corresponding current.  In particular, let
$\mc{M}^J(\Theta,\Theta')$ denote the moduli space of $J$-holomorphic
curves from $\Theta$ to $\Theta'$, where two such curves are
considered equivalent if they represent the same current in $\R\times
Y$, up to translation of the $\R$ coordinate.
\end{definition}

Given ECH generators $\Theta$ and $\Theta'$ with
$[\Theta]=[\Theta']=\Gamma$, the differential coefficient
$\langle\partial\Theta,\Theta'\rangle\in \Z/2$ is the mod 2 count of
$J$-holomorphic curves in $\mc{M}^J(\Theta,\Theta')$ with ``ECH
index'' equal to $1$.  The definition of the ECH index is not needed
in this paper and may be found in \cite{ir}.  If $J$ is generic,
then $\partial$ is well-defined and $\partial^2=0$, as shown in
\cite[\S7]{obg1}.  In this case we denote the chain complex by
$ECC_*(Y,\lambda,\Gamma;J)$. A symplectization-admissible almost
complex structure that is generic in this sense will be called {\em
  ECH-generic\/} here.  It turns out that the curves counted by the
ECH differential are embedded, except that they may include multiple
covers of $\R$-invariant cylinders.  The ECH index defines a relative
$\Z/d(c_1(\xi)+2\op{PD}(\Gamma))$ grading on the chain complex, where
$d$ denotes divisibility in $H^2(Y;\Z)/\op{Torsion}$.

\subsection{The isomorphism with Seiberg-Witten Floer cohomology}
\label{sec:isoswf}

Although the ECH differential depends on the choice of $J$, the
homology of the chain complex does not.  This follows from a much
stronger theorem of the second author \cite{e1,e2,e3,e4} asserting
that ECH is isomorphic to a version of Seiberg-Witten Floer cohomology
as defined by Kronheimer-Mrowka.  To be precise, there are three basic
versions of Seiberg-Witten Floer cohomology, denoted by
$\widehat{HM}^*$, $\check{HM}^*$, and $\overline{HM}^*$.  The first of
these is the one that is relevant to ECH; it assigns $\Z/2$-modules
$\widehat{HM}^*(Y,\frak{s})$ to each spin-c structure $\frak{s}$ on
$Y$, which have a relative $\Z/d(c_1(\frak{s}))$-grading. 

\begin{convention}
In this paper, all Seiberg-Witten Floer cohomology is defined with
$\Z/2$ coefficients (even though it can be defined over $\Z$, which is
the default coefficient system in \cite{km}).
\end{convention}

Recall that the set $\Spinc(Y)$ of spin-c structures on $Y$ is an
affine space over $H^2(Y;\Z)$, and the contact structure $\xi$
determines a distinguished spin-c structure $\frak{s}_\xi$.  With this
convention, the theorem is now that
\begin{equation}
\label{eqn:echswf}
ECH_*(Y,\lambda,\Gamma) \simeq \widehat{HM}^{-*}(Y,\frak{s}_\xi +
\op{PD}(\Gamma)),
\end{equation}
as relatively graded $\Z/2$-modules. (There is also an isomorphism
with $\Z$ coefficients \cite{e3}.)

It follows from scrutiny of the proof of \eqref{eqn:echswf}, together
with the invariance properties of $\widehat{HM}^*$, that the versions
of $ECH_*(Y,\lambda,\Gamma)$ defined using different almost complex
structures $J$ are canonically isomorphic to each other.  This point
is explained in detail in \cite{cc2}.  Thus it makes sense to talk
about $ECH_*(Y,\lambda,\Gamma)$ without referring to a choice of $J$.
Moreover, under this identification, the isomorphism
\eqref{eqn:echswf} is canonical.

At times it is convenient to ignore the homology class $\Gamma$ in the
definition of $ECH$, and simply define
\[
ECH_*(Y,\lambda) \eqdef \bigoplus_{\Gamma\in
  H_1(Y)}ECH_*(Y,\lambda,\Gamma).
\]
This is the homology of a chain complex $ECC_*(Y,\lambda;J)$ generated
by all admissible orbit sets, and by \eqref{eqn:echswf} this homology
is canonically isomorphic (as a relatively graded $\Z/2$-module) to
\[
\widehat{HM}^{-*}(Y)\eqdef
\bigoplus_{\frak{s}\in\Spinc(Y)}\widehat{HM}^{-*}(Y,\frak{s}).
\]

The proof of Theorem~\ref{thm:cc} makes use of two additional
structures on ECH: the action filtration and cobordism maps.  We now
explain these.

\subsection{The action filtration}
\label{sec:filtration}

If $\Theta=\{(\Theta_i,m_i)\}$ is an orbit set, its {\em symplectic
  action\/} or {\em length\/} is defined by
\begin{equation}
\label{eqn:length}
\mc{A}(\Theta) \eqdef \sum_i m_i \int_{\Theta_i}\gamma.
\end{equation}
The ECH differential for any (generic) symplectization-admissible $J$
decreases the action, i.e.\ if
$\langle\partial\Theta,\Theta'\rangle\neq 0$ then
$\mc{A}(\Theta)\ge\mc{A}(\Theta')$.  This is because if
$C\in\mc{M}^J(\Theta,\Theta')$, then $d\lambda|_C\ge 0$
everywhere\footnote{In fact if
  $\langle\partial\Theta,\Theta'\rangle\neq 0$ then the strict
  inequality $\mc{A}(\Theta)>\mc{A}(\Theta')$ holds, because
  $d\lambda$ vanishes identically on $C$ if and only if the image of
  $C$ is $\R$-invariant, in which case $C$ has ECH index zero and
  cannot contribute to the differential.}.  Thus for any real number
$L$, it makes sense to define the {\em filtered ECH\/}, denoted by $
ECH_*^{L}(Y,\lambda) $, to be the homology of the subcomplex
$ECC_*^L(Y,\lambda;J)$ of the ECH chain complex spanned by ECH
generators with action less than $L$.  It is shown in \cite{cc2} that
$ECH_*^{L}(Y,\lambda)$, just like $ECH_*(Y,\lambda)$, does not depend
on the choice of ECH-generic $J$.  However $ECH_*^L(Y,\lambda)$,
unlike the usual ECH, can change when one deforms the contact form
$\lambda$.

For $L<L'$ there is a map
\begin{equation}
\label{eqn:istar}
\imath^{L,L'}:ECH_*^{L}(Y,\lambda) \longrightarrow
ECH_*^{L'}(Y,\lambda)
\end{equation}
induced by the inclusion of chain complexes (for some given $J$,
although it is shown in \cite{cc2} that
\eqref{eqn:istar} does not depend on $J$).
The usual ECH is recovered as the direct limit
\begin{equation}
\label{eqn:edr}
ECH_*(Y,\lambda) = \lim_{L\to\infty}ECH_*^{L}(Y,\lambda).
\end{equation}
In particular, there is a natural map
\begin{equation}
\label{eqn:iL}
\imath^L:ECH_*^L(Y,\lambda) \longrightarrow ECH_*(Y,\lambda),
\end{equation}
again induced by an inclusion of chain complexes.

\subsection{Cobordism maps in ECH}
\label{sec:cobech}

Let $(Y_+,\lambda_+)$ and $(Y_-,\lambda_-)$ be closed oriented
3-manifolds with nondegenerate contact forms.  An {\em exact
  symplectic cobordism\/} from $(Y_+,\lambda_+)$ to $(Y_-,\lambda_-)$
is a compact symplectic 4-manifold $(X,\omega)$ with boundary
$\partial X = Y_+-Y_-$, for which there exists a $1$-form $\lambda$ on
$X$ such that $d\lambda=\omega$ on $X$ and
$\lambda|_{Y_\pm}=\lambda_\pm$.  A $1$-form $\lambda$ as above is
called a {\em Liouville form\/} for $(X,\omega)$. When we wish to
specify a Liouville form (which we usually do), we denote the exact
symplectic cobordism by $(X,\lambda)$, and we continue to write
$\omega=d\lambda$.

Note that our designation of the cobordism as ``from $Y_+$ to $Y_-$''
is natural from the perspective of symplectic geometry, but opposite
from the usual convention in Seiberg-Witten and Heegaard Floer
homology.  This is connected with the fact that embedded contact {\em
  homology\/} is identified with Seiberg-Witten Floer {\em cohomology\/}.

Now let $(X,\lambda)$ be an exact symplectic cobordism as
above.  This cobordism, like any smooth cobordism, induces
a map\footnote{Kronheimer-Mrowka define this map on the
  ``completed'' Seiberg-Witten Floer cohomology
  $\widehat{HM}^{\bullet}$.  However for $\widehat{HM}$, the completed
  and uncompleted cohomologies are the same.  Completion only makes a
  difference for the alternate versions $\check{HM}^*$ and
  $\overline{HM}^*$ of Seiberg-Witten Floer cohomology.

  Note also that if one uses coefficients in $\Z$ instead of $\Z/2$,
  then the signs in the cobordism map on $\widehat{HM}$ depend on a
  choice of ``homology orientation'' of $X$.  However one expects to
  be able to define cobordism maps on $ECH$ over $\Z$ without making
  such a choice, cf.\ \cite[Lem.\ A.14]{lw}.  Presumably an exact
  symplectic cobordism has a canonical homology orientation which
  makes the signs agree.}  of ungraded $\Z/2$-modules from the
Seiberg-Witten Floer cohomology of $Y_+$ to that of $Y_-$, which we
denote by
\begin{equation}
\label{eqn:swcob}
\widehat{HM}^*(X): \widehat{HM}^{*}(Y_+) \longrightarrow \widehat{HM}^{*}(Y_-).
\end{equation}

\begin{definition}
\label{def:PhiX}
Define
\begin{equation}
\label{eqn:gencob}
\Phi(X):
ECH_*(Y_+,\lambda_+) \longrightarrow ECH_*(Y_-,\lambda_-)
\end{equation}
to be the map on ECH obtained by composing the map \eqref{eqn:swcob}
on Seiberg-Witten Floer cohomology with the canonical isomorphism
\eqref{eqn:echswf} on both sides.
\end{definition}

It is natural to expect that the map \eqref{eqn:gencob} can be defined
directly, without using Seiberg-Witten theory, by suitably counting
holomorphic curves in the ``completion'' of $(X,\lambda)$.  The latter
is a noncompact symplectic manifold defined as follows.  To start, one
can find $\varepsilon>0$, a neighborhood $N_-$ of $Y_-$ in $X$, and an
identification $N_-\simeq [0,\varepsilon)\times Y_-$, such that
$\lambda=e^s\lambda_-$ on $N_-$, where $s$ denotes the
$[0,\varepsilon)$ coordinate.  The requisite map
$[0,\varepsilon)\times Y_-\to X$ is obtained using the flow starting
at $Y_-$ of the unique vector field $V$ on $X$ such that
$\imath_V\omega=\lambda$.  Likewise, a neighborhood $N_+$ of $Y_+$ in
$X$ can be identified with $(-\varepsilon,0]\times Y_+$ so that
$\lambda=e^s\lambda_+$ on $N_+$.  Using these identifications, one can
then glue symplectization ends to $X$ to obtain the {\em completion\/}
\begin{equation}
\label{eqn:completion}
\overline{X} \eqdef ((-\infty,0]\times Y_-) \cup_{Y_-} X \cup_{Y_+}
([0,\infty)\times Y_+).
\end{equation}
Note for reference later that the Liouville form $\lambda$ on $X$
canonically extends to a $1$-form on $\overline{X}$ which equals
$e^s\lambda_\pm$ on the ends.

\begin{definition}
\label{def:cobadm}
An almost complex structure $J$ on $\overline{X}$ is {\em
  cobordism-admissible\/} if it is
$\omega$-compatible\footnote{Everything we describe below should still
  be possible if one weakens the $\omega$-compatible condition here to
  $\omega$-tame.  However, because the papers relating Seiberg-Witten
  Floer cohomology to ECH use compatible almost complex structures, we
  will stick with the latter to avoid confusion.} on $X$, and if it
agrees with symplectization-admissible almost complex structures $J_+$
for $\lambda_+$ on $[0,\infty)\times Y_+$ and $J_-$ for $\lambda_-$ on
$(-\infty,0]\times Y_-$.
\end{definition}

Given a cobordism-admissible $J$, and given (not necessarily
admissible) orbit sets $\Theta^+=\{(\Theta_i^+,m_i^+)\}$ in $Y_+$ and
$\Theta^-=\{(\Theta_j^-,m_i^-)\}$ in $Y_-$, we define a
``$J$-holomorphic curve in $\overline{X}$ from $\Theta^+$ to
$\Theta^-$'' analogously to Definition~\ref{def:Jhol}, and denote the
moduli space of such curves by $\mc{M}^J(\Theta^+,\Theta^-)$, where
two such curves are considered equivalent if they represent the same
current in $\overline{X}$.

One would now like to define the map \eqref{eqn:gencob} by choosing a
generic cobordism-admissible $J$ and suitably counting $J$-holomorphic
curves in $\overline{X}$ as above with ECH index 0, so as to define a
chain map between the ECH chain complexes which induces the map
\eqref{eqn:gencob} on homology.  (In general one also needs to include
contributions from ``broken'' $J$-holomorphic curves, see
\S\ref{sec:exact}.)  An important consequence of such a construction
would be that the map \eqref{eqn:gencob} respects the action
filtrations, i.e.\ is induced by a chain map which does not increase
the action filtration.  The reason is that if $C$ is any holomorphic
curve in $\mc{M}^J(\Theta^+,\Theta^-)$, then by Stokes' theorem and
the exactness of the cobordism we have
\begin{equation}
\label{eqn:stokes}
\mc{A}(\Theta^+)-\mc{A}(\Theta^-) = \int_{C\cap [0,\infty)\times
  Y_+}d\lambda_+ + \int_{C\cap X}\omega + \int_{C\cap
  (-\infty,0]\times Y_-}d\lambda_-,
\end{equation}
and all of the integrands on the right hand side are pointwise
nonnegative by our assumptions on $J$.

Unfortunately it is not currently known how to define the map
\eqref{eqn:gencob} in terms of holomorphic curves as above.  The
difficulty is that, as explained in \cite[\S5]{ir}, the
compactifications of the relevant moduli spaces of holomorphic curves
can include broken curves with negative index multiply covered
components, and it is not clear in general what these should
contribute to the count (although examples show that such broken
curves must sometimes make nonzero contributions).  However we can
still use Seiberg-Witten theory to show that the map
\eqref{eqn:gencob} respects the action filtrations (in a slightly
weaker sense than above), and enjoys some other useful properties
which would follow from a definition in terms of holomorphic curves.
The precise statement uses filtered ECH and is given in
Theorem~\ref{thm:cob} below.

The basic idea of the proof of Theorem~\ref{thm:cob} is to perturb the
Seiberg-Witten equations on $\overline{X}$ using a large multiple of
the symplectic form, much as in the proof of \eqref{eqn:echswf}, and
to show that with such a perturbation, Seiberg-Witten solutions that
contribute to the cobordism map \eqref{eqn:gencob} give rise to
(possibly broken) holomorphic curves.  The main analytical machinery
is adapted from the proof of \eqref{eqn:echswf} in \cite{e1,e4}.
Nonetheless the detailed proof is still long, so we have deferred it
to the sequel \cite{cc2}.

\subsection{Legendrian surgery}

Returning finally to the chord conjecture, let $(Y_0,\lambda_0)$ be a
closed oriented 3-manifold with a contact form, and let $\mc{K}$ be a
Legendrian knot in $(Y_0,\lambda_0)$.  The contact structure
determines a framing of $\mc{K}$, which we denote by $tb(\mc{K})$.
Let $Y_1$ denote the 3-manifold obtained by surgery on $\mc{K}$ with
framing $tb(\mc{K})-1$.  The surgery procedure determines a smooth
cobordism $X$ from $Y_1$ to $Y_0$.  As was shown in \cite{weinstein}
and as we review in \S\ref{sec:ls}, the 3-manifold $Y_1$ has a natural
contact structure, which can be expressed as the kernel of a contact
form $\lambda_1$ such that $X$ has the structure of an exact
symplectic cobordism from $(Y_1,\lambda_1)$ to $(Y_0,\lambda_0)$.
Moreover, as is familiar from the work of Bourgeois-Ekholm-Eliashberg
\cite{bee} on Legendrian surgery in contact homology, the contact form
$\lambda_1$ can be chosen so that, modulo ``long'' Reeb orbits, one
has:
\begin{description}
\item{(*)}
The Reeb orbits of $\lambda_1$ correspond to the Reeb orbits of
$\lambda_0$, together with cyclic words in the Reeb chords of
$\mc{K}$.
\end{description}
In particular, if $\mc{K}$ has no Reeb chord, then
$\lambda_1$ and $\lambda_0$ have the same ``short'' Reeb orbits.

The idea of the proof of the chord conjecture is to use the preceding
observation, together with Theorem~\ref{thm:cob} regarding the
properties of ECH cobordism maps, to show that if there is no Reeb
chord then the ECH cobordism map
\begin{equation}
\label{eqn:surcob}
\Phi(X): ECH_*(Y_1,\lambda_1) \longrightarrow ECH_*(Y_0,\lambda_0)
\end{equation}
induced by the Legendrian surgery cobordism is an isomorphism.  Note
that this is what one would expect by analogy with a very special case
of the aforementioned work of Bourgeois-Ekholm-Eliashberg.

But the map \eqref{eqn:surcob} cannot be an isomorphism, because this
would contradict results of Kronheimer-Mrowka, namely:

\begin{lemma}
\label{lem:ni}
If $Y_1$ is obtained from a closed oriented 3-manifold $Y_0$ by
surgery along a knot $\mc{K}$, and if $X$ denotes the corresponding
smooth cobordism from $Y_1$ to $Y_0$, then the induced map on
Seiberg-Witten Floer cohomology with $\Z/2$ coefficients,
\begin{equation}
\label{eqn:ni}
\widehat{HM}^*(X):\widehat{HM}^*(Y_1) \longrightarrow
\widehat{HM}^*(Y_0),
\end{equation}
is not an isomorphism.
\end{lemma}

\begin{proof}
  It follows from \cite[Thm.\ 42.2.1]{km}, see also \cite{bloom,kmos},
  that there is an exact triangle
\[
\cdots
\longrightarrow
\widehat{HM}^*(Y_2) \longrightarrow \widehat{HM}^*(Y_1)
\stackrel{\widehat{HM}^*(X)}{\longrightarrow} \widehat{HM}^*(Y_0)
\longrightarrow \widehat{HM}^*(Y_2) \longrightarrow \cdots
\]
where $Y_2$ is obtained from $Y_0$ by a certain different surgery
along $\mc{K}$.  Note that the exact triangle was only proved over
$\Z/2$, so it is fortunate that we are using $\Z/2$ coefficients
everywhere.  Now the exact triangle implies that if \eqref{eqn:ni}
were an isomorphism, then $\widehat{HM}^*(Y_2)$ would vanish.  However
the latter is nontrivial, because it follows from \cite[Cor.\
35.1.4]{km} that for any 3-manifold $Y$, if $\frak{s}$ is a torsion
spin-c structure on $Y$ (these always exist), then
$\widehat{HM}^*(Y,\frak{s})$ is infinitely generated.  This is proved
in \cite{km} with $\Z$ coefficients, which immediately implies the
statement with $\Z/2$ coefficients.
\end{proof}

There are two wrinkles in the above argument.  First, statement (*) is
true only for Reeb orbits whose action is not too large, where the
definition of ``large'' depends on the details of the Legendrian
surgery construction.  However one can modify $\lambda_1$ so as to
make the corresponding upper bound on the action arbitrary large, see
Lemma~\ref{lem:ro} below.  Moreover the different versions of
$(Y_1,\lambda_1)$ fit into a sequence of exact cobordisms. As a
result, by making appropriate use of the cobordism maps on filtered
ECH, we can still show that if there is no Reeb chord then the ECH
cobordism map \eqref{eqn:surcob} induced by the Legendrian surgery is
an isomorphism.

Second, the above argument only makes sense if the contact form
$\lambda_0$ (and with it the contact form $\lambda_1$) is
nondegenerate, so that its ECH chain complex is well-defined.  A
priori there could exist a degenerate contact form and a Legendrian
knot with no Reeb chord, such that for any nondegenerate perturbation
of the contact form the knot does have a Reeb chord.  To deal with
this issue, we will show that when $\lambda_0$ is nondegenerate, there
exists a Reeb chord with an upper bound on its symplectic action,
given by a quantitative measure of the failure of the cobordism map
\eqref{eqn:surcob} to be an isomorphism.  The precise statement is
given in Theorem~\ref{thm:nondegenerate} below.  The aforementioned
upper bound depends ``continuously'' on the contact form, as shown in
Proposition~\ref{prop:continuity}.  It then follows from a compactness
argument that the chord conjecture holds in the degenerate case as
well.

\paragraph{Contents of the rest of the paper.}
In \S\ref{sec:exact} we give the precise statement of
Theorem~\ref{thm:cob} on the existence and properties of maps on
(filtered) ECH induced by exact symplectic cobordisms.  In
\S\ref{sec:ls}--\S\ref{sec:ccdegenerate} we use Theorem~\ref{thm:cob}
as a ``black box'' to prove the chord conjecture.  The formal proof of
the chord conjecture is put together at the end of
\S\ref{sec:ccdegenerate}.  In the sequel \cite{cc2} we use
Seiberg-Witten theory to prove Theorem~\ref{thm:cob}.

\paragraph{Acknowledgments.} We thank Jonathan Bloom, Tobias Ekholm,
Yasha Eliashberg, Ko Honda, Dusa McDuff, Tomasz Mrowka, and Ivan Smith
for helpful discussions.  The first author was partially supported by
NSF grant DMS-0806037.  The second author was partially supported by
the Clay Mathematics Insitute, the Mathematical Sciences Research
Institute, and the NSF.  Both authors thank MSRI, where this work was
carried out, for its hospitality.

\section{ECH and exact symplectic cobordisms}
\label{sec:exact}

We now state the theorem on the existence and properties of maps on (filtered)
ECH induced by exact symplectic cobordisms.

We need some preliminary definitions.  Below, let $(X,\lambda)$ be an
exact symplectic cobordism from $(Y_+,\lambda_+)$ to
$(Y_-,\lambda_-)$, and assume that the contact forms $\lambda_\pm$ are
nondegenerate.  Fix a cobordism-admissible almost complex structure
$J$ on $\overline{X}$ which restricts to symplectization-admissible
almost complex structures $J_+$ on $[0,\infty)\times Y_+$ and $J_-$ on
$(-\infty,0]\times Y_-$, as in Definition~\ref{def:cobadm}.

\paragraph{Broken curves.}

Let $\Theta^+$ and $\Theta^-$ be (not
necessarily admissible) orbit sets in $(Y_+,\lambda_+)$ and
$(Y_-,\lambda_-)$ respectively.

\begin{definition}
\label{def:broken}
A {\em broken $J$-holomorphic curve from $\Theta^+$ to $\Theta^-$\/}
is a collection of holomorphic curves $\{C_k\}_{1\le k\le N}$, and
(not necessarily admissible) orbit sets $\Theta^{k+}$ and
$\Theta^{k-}$ for each $k$, such that there exists
$k_0\in\{1,\ldots,N\}$ such that:
\begin{itemize}
\item
$\Theta^{k+}$ is an orbit set in $(Y_+,\lambda_+)$ for each $k\ge k_0$;
$\Theta^{k-}$ is an orbit set in $(Y_-,\lambda_-)$ for each $k\le k_0$; 
$\Theta^{N+}=\Theta^+$; $\Theta^{1-}=\Theta^-$; and
$\Theta^{k-}=\Theta^{k-1,+}$ for each $k>1$.
\item
If $k>k_0$ then $C_k\in\mc{M}^{J_+}(\Theta^{k+},\Theta^{k-})$;
$C_{k_0}\in\mc{M}^J(\Theta^{k_0,+},\Theta^{k_0,-})$; 
and if $k<k_0$ then $C_k\in\mc{M}^{J_-}(\Theta^{k+},\Theta^{k-})$.
\item
If $k\neq k_0$ then $C_k$ is not $\R$-invariant (as a current).
\end{itemize}
Let $\overline{\mc{M}^J(\Theta^+,\Theta^-)}$ denote the moduli space
of broken $J$-holomorphic curves from $\Theta^+$ to $\Theta^-$ as above.
\end{definition}

Note that $\mc{M}^J(\Theta^+,\Theta^-)$ is a subset of
$\overline{\mc{M}^J(\Theta^+,\Theta^-)}$ corresponding to broken
curves as above in which $N=1$ (and it is perhaps a misnomer to call
such curves ``broken'').

\paragraph{Product cylinders.}
If the cobordism $(X,\lambda)$ and the almost complex structure $J$ on
$\overline{X}$ are very special, then $X$ may contain regions that
look like pieces of a symplectization, in the following sense:

\begin{definition}
\label{def:PR}
  A {\em product region\/} in $X$ is the image of an embedding
  $[s_-,s_+]\times Z \to X$, where $s_-<s_+$ and $Z$ is an open
  3-manifold, such that:
\begin{itemize}
\item $\{s_\pm\}\times Z$ maps to $Y_\pm$, and $(s_-,s_+)\times Z$
  maps to the interior of $X$.
\item The pullback of the Liouville form $\lambda$ to $[s_-,s_+]\times
  Z$ has the form $e^s\lambda_0$, where $s$ denotes the $[s_-,s_+]$
  coordinate, and $\lambda_0$ is a contact form on $Z$.
\item The pullback of the almost complex structure $J$ to
  $[s_-,s_+]\times Z$ has the following properties:
\begin{itemize}
\item
The restriction of $J$ to $\Ker(\lambda_0)$ is independent of $s$.
\item
$J(\partial/\partial s)=f(s)R_0$, where $f$ is a positive function of
$s$ and $R_0$ denotes the Reeb vector field for $\lambda_0$.
\end{itemize}
\end{itemize}
\end{definition}

Given a product region as above, the embedded Reeb orbits of
$\lambda_\pm$ in $\{s_\pm\}\times Z$ are identified with the embedded
Reeb orbits of $\lambda_0$ in $Z$.  If $\gamma$ is such a Reeb orbit,
then we can form a $J$-holomorphic cylinder in $\overline{X}$ by
taking the union of $[s_-,s_+]\times \gamma$ in $[s_-,s_+]\times Z$
with $(-\infty,0]\times \gamma$ in $(-\infty,0]\times Y_-$ and
$[0,\infty)\times\gamma$ in $[0,\infty)\times Y_+$.

\begin{definition}
\label{def:PC}
We call a $J$-holomorphic cylinder as above a {\em product
  cylinder.\/}
\end{definition}

\paragraph{Composition of cobordisms.} 
If $(X_1,\lambda_1)$ is an exact symplectic cobordism from
$(Y_+,\lambda_+)$ to $(Y_0,\lambda_0)$, and if $(X_2,\lambda_2)$ is an
exact symplectic cobordism from $(Y_0,\lambda_0)$ to
$(Y_-,\lambda_-)$, then we can compose them to obtain an exact
symplectic cobordism $(X_2\circ X_1,\lambda)$ from $(Y_+,\lambda_+)$
to $(Y_-,\lambda_-)$.  Here $X_2\circ X_1$ is obtained by gluing $X_1$
and $X_2$ along $Y_0$ analogously to \eqref{eqn:completion}, and
$\lambda|_{X_i}=\lambda_i$ for $i=1,2$.

\paragraph{Homotopy of cobordisms.}
Two exact symplectic cobordisms $(X,\omega_0)$ and $(X,\omega_1)$ from
$(Y_+,\lambda_+)$ to $(Y_-,\lambda_-)$ with the same underlying
four-manifold $X$ are {\em homotopic\/} if there is a one-parameter family
of symplectic forms $\{\omega_t\mid t\in[0,1]\}$ on $X$ such that
$(X,\omega_t)$ is an exact symplectic cobordism from $(Y_+,\lambda_+)$
to $(Y_-,\lambda_-)$ for each $t\in[0,1]$.

\paragraph{Scaling.}
If $\lambda$ is a nondegenerate contact form on $Y$, and if $c$ is a
positive constant, then there is a canonical ``scaling'' isomorphism
\begin{equation}
\label{eqn:scaling}
s:ECH_*^{L}(Y,\lambda) \stackrel{\simeq}{\longrightarrow}
 ECH_*^{cL}(Y,c\lambda).
\end{equation}
To see this, observe that the chain complexes on both sides have the
same generators.  Moreover, an ECH-generic almost complex structure
$J$ for $\lambda$ induces a symplectization-admissible almost complex
structure $J^c$ for $c\lambda$, such that $J$ and $J^c$ agree when
restricted to the contact planes $\xi$.  The self-diffeomorphism of
$\R\times Y$ sending $(s,y)\mapsto (cs,y)$ then induces a bijection
between $J$-holomorphic curves and $J^c$-holomorphic curves.  So with
these choices, the canonical identification of generators is an
isomorphism of chain complexes.  Moreover, it is shown in \cite{cc2}
that the resulting isomorphism \eqref{eqn:scaling} does not depend on
$J$ (under the canonical isomorphisms between the versions of ECH
defined using different almost complex structures).

\begin{theorem}
\label{thm:cob}
Let $(Y_+,\lambda_+)$ and $(Y_-,\lambda_-)$ be closed oriented
3-manifolds with nondegenerate contact forms.  Let $(X,\lambda)$ be an exact
symplectic cobordism from $(Y_+,\lambda_+)$ to $(Y_-,\lambda_-)$.
Then there exist maps (of ungraded $\Z/2$-modules)
\begin{equation}
\label{eqn:PhiL}
\Phi^L(X,\lambda): ECH_*^{L}(Y_+,\lambda_+) \longrightarrow
ECH_*^{L}(Y_-,\lambda_-)
\end{equation}
for each real number $L$, such that:
\begin{description}
\item{(Homotopy Invariance)} The map $\Phi^L(X,\lambda)$ depends only on $L$
  and the homotopy class of $(X,\omega)$.
\item{(Inclusion)} If $L<L'$ then the following diagram commutes:
\[
\begin{CD}
ECH_*^{L}(Y_+,\lambda_+) @>{\Phi^L(X,\lambda)}>> ECH_*^{L}(Y_-,\lambda_-) \\
@VV{\imath^{L,L'}}V @VV{\imath^{L,L'}}V \\
ECH_*^{L'}(Y_+,\lambda_+) @>{\Phi^{L'}(X,\lambda)}>> ECH_*^{L'}(Y_-,\lambda_-).
\end{CD}
\]
\item{(Direct Limit)}
\[
\lim_{L\to\infty}\Phi^L(X,\lambda) = \Phi(X): ECH_*(Y_+,\lambda_+) \longrightarrow
ECH_*(Y_-,\lambda_-),
\]
where $\Phi(X)$ is as in Definition~\ref{def:PhiX}.
\item{(Composition)}
If $(X,\lambda)$ is the composition of $(X_2,\lambda_2)$ and
$(X_1,\lambda_1)$ as above
with $\lambda_0$ nondegenerate, then
\[
\Phi^L(X_2\circ X_1,\lambda) = \Phi^L(X_2,\lambda_2) \circ
\Phi^L(X_1,\lambda_1).
\] 
\item{(Scaling)} If $c$ is a positive constant then the following
  diagram commutes:
\[
\begin{CD}
ECH_*^{L}(Y_+,\lambda_+) @>{\Phi^L(X,\lambda)}>> ECH_*^{L}(Y_-,\lambda_-) \\
@VV{s}V @VV{s}V \\
ECH_*^{cL}(Y_+,c\lambda_+) @>{\Phi^{cL}(X,c\lambda)}>> ECH_*^{cL}(Y_-,c\lambda_-).
\end{CD}
\]
\item{(Holomorphic Curves)} Let $J$ be a cobordism-admissible almost
  complex structure on $\overline{X}$ such that $J_+$ and $J_-$ are
  ECH-generic.  Then there exists a (noncanonical) chain map
\[
\hat{\Phi}^L : ECC_*^L(Y_+,\lambda_+,J_+) \longrightarrow
ECC_*^L(Y_-,\lambda_-,J_-)
\]
inducing $\Phi^L(X,\lambda)$, such that if $\Theta^+$ and $\Theta^-$
are ECH generators for $(Y_+,\lambda_+)$ and $(Y_-,\lambda_-)$
respectively with action less than $L$, then:
\begin{description}
\item{(i)} If there are no broken $J$-holomorphic curves in
  $\overline{X}$ from $\Theta^+$ to $\Theta^-$, then
  $\langle \hat{\Phi}^L\Theta^+,\Theta^-\rangle=0$.
\item{(ii)} If the only broken $J$-holomorphic curve in $\overline{X}$
  from $\Theta^+$ to $\Theta^-$ is a union of covers of product
  cylinders, then $\langle \hat{\Phi}^L\Theta^+,\Theta^-\rangle=1$.
\end{description}
\end{description}
\end{theorem}

\begin{example}
  For any three-manifold with a nondegenerate contact form, the empty
  set of Reeb orbits is a cycle in the ECH chain complex (whose
  homology class in ECH
  corresponds to the ``contact invariant'' in Seiberg-Witten Floer
  cohomology, see \cite{e5}).
  If $\emptyset_\pm$ denotes the empty set of Reeb orbits, regarded as a
  generator of the ECH chain complex for $Y_\pm$, then it follows from
  the Holomorphic Curves axiom that
\[
\Phi^L(X,\lambda): [\emptyset_+] \longmapsto [\emptyset_-].
\] 
The reason is that for any cobordism-admissible almost complex structure $J$, by
\eqref{eqn:stokes} there is a unique $J$-holomorphic curve with no
positive end, namely the empty holomorphic curve.
\end{example}

\begin{remark}
  Theorem~\ref{thm:cob} has applications beyond the chord conjecture,
  for example to symplectic embedding obstructions \cite{qech}.  The
  Scaling axiom is not needed for the proof of the chord conjecture,
  but is useful in these other applications.
\end{remark}

\section{Legendrian surgery}
\label{sec:ls}

We now explain the details of the Legendrian surgery construction we
will use.  In particular we define a sequence of Legendrian surgeries,
related to each other by exact symplectic cobordisms, in which the
Reeb vector field is increasingly well-behaved.

To begin, recall that a {\em Liouville vector field\/} on
a symplectic manifold $(X,\omega)$ is a vector field $V$ such that
$\mc{L}_V\omega = \omega$.  A Liouville vector field $V$ is equivalent
to a $1$-form $\lambda$ such that $d\lambda=\omega$, via the equation
$\lambda=\imath_V\omega$.

If $Y$ is a hypersurface in $(X,\omega)$ transverse to a Liouville
vector field $V$, then $\lambda_Y\eqdef \lambda|_Y$ is a contact form
on $Y$.  Now let $Y'$ be another hypersurface transverse to $V$,
and suppose that the time $t$ flow of $V$ defines a diffeomorphism
$\phi:Y\to Y'$, where $t$ is some function on $Y$.  Then the contact
forms on $Y$ and $Y'$ are related by
\begin{equation}
\label{eqn:et}
\phi^*\lambda_{Y'} = e^t\lambda_Y.
\end{equation}

With the above preliminaries out of the way, consider now a closed
oriented 3-manifold $Y_0$ with a nondegenerate contact form
$\lambda_0$.  Let $\mc{K}$ be a Legendrian knot in $(Y_0,\lambda_0)$.
Let $Y_1$ be the 3-manifold obtained by surgery along $\mc{K}$ with
framing $tb(\mc{K})-1$.

\begin{proposition}
\label{prop:ls}
There exist:
\begin{itemize}
\item a nondegenerate contact form $\lambda_1$ on $Y_1$,
\item an exact symplectic cobordism $(X,\lambda)$ from
  $(Y_1,\lambda_1)$ to $(Y_0,\lambda_0)$; let $V$ denote its
  associated associated Liouville vector field;
\item a compact hypersurface $Y_{1/n}$ in $X$ transverse to $V$ for each
  positive integer $n$,
\item and neighborhoods $U_{1/n}$ of $\mc{K}$ in $Y_0$ with
  $U_{1/n}\supset U_{1/(n+1)}$ and $\bigcap_{n=1}^\infty
  U_{1/n}=\mc{K}$,
\end{itemize}
with the following properties:
\begin{description}
\item{(a)} The induced contact form $\lambda_{1/n}$ on $Y_{1/n}$ is
  nondegenerate.
\item{(b)} The negative time flow of $V$ induces a diffeomorphism
  $Y_{1/n}\stackrel{\simeq}{\to} Y_{1/(n+1)}$.  (The flow time varies
  over $Y_{1/n}$.)
\item{(c)} The time $-1/n$ flow of $V$, call it $\phi_{-1/n}$, is
  defined on all of $Y_{1/n}$.
There is a subset $\widetilde{U}_{1/n}$ of $Y_{1/n}$ such that
\[
\phi_{-1/n}(Y_{1/n}\setminus \widetilde{U}_{1/n}) = Y_0\setminus
  U_{1/n}.
\]
\item{(d)} The Reeb vector field on $(Y_{1/n},\lambda_{1/n})$ has no
  closed orbits contained entirely within $\widetilde{U}_{1/n}$.
\end{description}
\end{proposition}

\begin{proof}
  The idea for building the cobordism $(X,\lambda)$ is to start with
  the exact symplectic cobordism $([0,1]\times Y_0,e^s\lambda_0)$,
  where $s$ denotes the $[0,1]$ coordinate.  One then attaches a
  2-handle with an appropriate Liouville form to $\{1\}\times Y_0$ in
  a neighborhood of $\{1\}\times\mc{K}$.  We proceed in four steps.

  {\em Step 1.\/} We first describe a model for the handle attachment,
  following \cite{weinstein}.  Consider $\R^4$ with coordinates
  $q_1,q_2,p_1,p_2$ and the symplectic form $\omega=\sum_{i=1}^2
  dp_i\, dq_i$.  Define a Liouville vector field on $\R^4$ by
\[
V = \sum_{i=1}^2 \left(-p_i\frac{\partial}{\partial p_i} +
  2q_i\frac{\partial}{\partial q_i}\right).
\]
Consider the hypersurface
\[
Y = (p_1^2+p_2^2=1)\subset \R^4,
\]
regarded as the boundary of $(p_1^2+p_2^2\ge 1)$.  The Liouville
vector field $V$ is transverse to $Y$ and so induces a contact form on
$Y$.  With respect to this contact form, the circle
\[
C=(q_1=q_2=0,p_1^2+p_2^2=1)
\]
is a Legendrian knot.  To the region $(p_1^2+p_2^2\ge 1)$ we now
attach the $2$-handle consisting of the subset of $\R^4$ where
\[
p_1^2+p_2^2\le 1, \quad \quad q_1^2+q_2^2\le \varepsilon
\]
for some $\varepsilon>0$.  The boundary of the region with the
2-handle attached has a corner where $p_1^2+p_2^2=1$ and $q_1^2+q_2^2
= \varepsilon$.  To round the corner, we replace the boundary
hypersurface $q_1^2+q_2^2=\varepsilon$ of the handle with a nearby
hypersurface, staying within the region
$\varepsilon/2<q_1^2+q_2^2\le\varepsilon$, and defined by an
equation of the form
\begin{equation}
\label{eqn:fqp}
f(q_1^2+q_2^2,p_1^2+p_2^2)=0,
\end{equation}
where at each point on the zero set of $f$ we have $\partial
f/\partial x > 0$ and $\partial f/\partial y \le 0$.  The
boundary of the region with the 2-handle attached is then a smooth
hypersurface $Y'$ which is transverse to the Liouville vector field
$V$.

{\em Step 2.\/} We now pass from the model case to the case of
interest.  By \cite[Prop.\ 4.2]{weinstein}, there is a diffeomorphism
of a neighborhood $N$ of $C$ in $\R^4$ with a neighborhood of
$\{1\}\times \mc{K}$ in the symplectization $\R\times Y_0$, which
respects the symplectic forms and Liouville vector fields and locally
identifies the hypersurface $Y$ in $\R^4$ with the hypersurface
$\{1\}\times Y_0$ in $\R\times Y_0$.  If $\varepsilon>0$ is
sufficiently small, then $N\cap Y$ will contain the region in $Y$ to
which the $2$-handle in $\R^4$ is attached.  We then use the above
diffeomorphism to attach the $2$-handle described above in $\R^4$,
with its symplectic form and Liouville vector field, to $[0,1]\times
Y_0$.  We now provisionally define $X$ to be the resulting exact
symplectic cobordism, and $(Y_1,\lambda_1)$ to be its positive
boundary with the induced contact form.  It is not hard to check that
as a smooth 3-manifold, $Y_1$ is obtained from $Y_0$ by surgery on
$\mc{K}$ with framing $tb(\mc{K})-1$.  We also define
$\widetilde{U}_1$ to be the part of $Y_1$ in the handle, and
$U_1=Y_0\setminus \phi_{-1}(Y_1\setminus\widetilde{U}_1)$.  That is,
$U_1\subset Y_0$ corresponds to the subset of $\{1\}\times Y_0$ to
which the handle is attached.

{\em Step 3.\/} We now check that the Reeb vector field $R_1$ on $Y_1$
has the required properties.  We first show that $R_1$ has no closed
orbit contained in $\widetilde{U}_1$. On $\widetilde{U}_1$, in terms
of the coordinates on $\R^4$, the Reeb vector field $R_1$ is parallel
to the Hamiltonian vector field associated to the function
\eqref{eqn:fqp}.  Thus
\[
gR_1 = \frac{\partial f}{\partial
    x}\left(q_1\frac{\partial}{\partial p_1} +
    q_2\frac{\partial}{\partial p_2}\right) -
  \frac{\partial f}{\partial y}\left(p_1\frac{\partial}{\partial q_1}
    + p_2\frac{\partial}{\partial q_2}\right),
\]
where $g$ is some positive function on $\widetilde{U}_1$.  Now define
another function $h$ on $\widetilde{U}_1$ by
\[
h \eqdef p_1q_1 + p_2q_2.
\]
We then compute that
\begin{equation}
\label{eqn:Rh}
R_1(h)>0
\end{equation}
on all of $\widetilde{U}_1$.  It follows immediately that $R_1$ has
no closed orbit contained in $\widetilde{U}_1$.

Next we consider nondegeneracy of $\lambda_1$.  By construction,
$\lambda_1$ is a constant multiple of $\lambda_0$ outside of
$\widetilde{U}_1$.  Since $\lambda_0$ was assumed nondegenerate, it
follows that any Reeb orbit for $\lambda_1$ that avoids the region
$\widetilde{U}_1$ is nondegenerate as well.  Consequently we can make
$\lambda_1$ nondegenerate by perturbing it (specifically, multiplying
it by a positive function close to $1$) in $\widetilde{U}_1$.  By
equation \eqref{eqn:et}, such a perturbation of $\lambda_1$ can be
effected by perturbing the hypersurface $\widetilde{U}_1$ in the
definition of $X$.  If this perturbation is sufficiently $C^1$-small,
then \eqref{eqn:Rh} will still hold, so the Reeb vector field of
$\lambda_1$ will now have all of the required properties.

{\em Step 4.\/} The hypersurface $Y_{1/n}\subset X$ is now defined to
be the positive boundary of the region obtained by starting with
$[0,1/n]\times Y_0$ and attaching a taller and thinner $2$-handle.
This handle is obtained by starting with the subset of $\R^4$ where
\[
p_1^2+p_2^2 \le e^{2(1-1/n)}, \quad\quad q_1^2+q_1^2 \le
2^{1-n}\varepsilon,
\]
then rounding corners as before and perturbing if necessary to make
$\lambda_{1/n}$ nondegenerate.  Finally, one defines
$\widetilde{U}_{1/n}$ to be the part of $Y_{1/n}$ in
the handle, and
$U_{1/n}=Y_0\setminus\phi_{-1/n}(Y_1\setminus\widetilde{U}_1)$.
\end{proof}

A basic consequence of the above construction is the following:

\begin{lemma}
\label{lem:ro}
Suppose $\mc{K}$ has no Reeb chord with action $\le L$.  Then for all
$n$ sufficiently large:
\begin{description}
\item{(a)} The Reeb orbits of $\lambda_{1/n}$ with action $<e^{1/n}L$
  avoid the region $\widetilde{U}_{1/n}$.
\item{(b)} $\phi_{-1/n}$ defines a bijection from the Reeb orbits of
  $\lambda_{1/n}$ with action $<e^{1/n}L$ to the Reeb orbits
  of $\lambda_0$ with action $<L$.
\end{description}
\end{lemma}

\begin{proof}
  Suppose $\{n_k\}_{k=1,2,\ldots}$ is an increasing sequence of
  positive integers and that for each $k$ there exists a Reeb orbit
  $\gamma_k$ for $\lambda_{1/n_k}$ of action $<e^{1/n_k}L$
  intersecting $\widetilde{U}_{1/n_k}$.  Then for each $k$, the set
  $\phi_{-1/n_k}(\gamma_k\cap(Y_{1/n_k}\setminus\widetilde{U}_{1/n_k}))$
  is a union of Reeb trajectories of $\lambda_0$ starting and ending
  on the boundary of $U_{1/n_k}$ with total action less than $L$.
  Picking one of these trajectories for each $k$, we can pass to a
  subsequence so that these trajectories converge to a Reeb chord with
  action $\le L$.  This proves (a).  Similarly, if $n$ is sufficiently
  large then the Reeb orbits of $\lambda_0$ with action $<L$ avoid the
  region $U_{1/n}$.  This together with (a) implies (b).
\end{proof}

\section{The chord conjecture: nondegenerate case}
\label{sec:ccnondegenerate}

We now prove the chord conjecture, Theorem~\ref{thm:cc}, in the case
when the contact form $\lambda_0$ is nondegenerate.  Below, we use the
notation from the Legendrian surgery construction in
Proposition~\ref{prop:ls}.  Also, to shorten the notation we write
$H_*(Y,\lambda)$ to denote $ECH_*(Y,\lambda)$, and $H_*^L(Y,\lambda)$
to denote $ECH_*^{L}(Y,\lambda)$.

Observe that by the construction in \S\ref{sec:ls}, the exact
symplectic cobordism $(X,\lambda)$ contains an exact symplectic
cobordism from $(Y_{1/n},\lambda_{1/n})$ to $(Y_0,\lambda_0)$, call
this $X_{1/n}$.  The main lemma is now:
\begin{lemma}
\label{lem:main}
Let $L>0$.  Suppose that $\mc{K}$ has no Reeb chord of action $\le L$.
Then for all $n$ sufficiently large, the cobordism map
\begin{equation}
\label{eqn:cob1}
\Phi^{e^{1/n}L}(X_{1/n},\lambda): H_*^{e^{1/n}L}(Y_{1/n},\lambda_{1/n}) \longrightarrow
H_*^{e^{1/n}L}(Y_0,\lambda_0)
\end{equation}
is the composition of an isomorphism
\begin{equation}
\label{eqn:iso1}
H_*^{e^{1/n}L}(Y_{1/n},\lambda_{1/n}) \stackrel{\simeq}{\longrightarrow}
H_*^L(Y_0,\lambda_0)
\end{equation}
with the inclusion-induced map $\imath^{L,e^{1/n}L}: H_*^L(Y_0,\lambda_0)\to
H_*^{e^{1/n}L}(Y_0,\lambda_0)$.
\end{lemma}

\begin{proof}
  By Lemma~\ref{lem:ro}, if $n$ is sufficiently large, then the Reeb
  orbits for $\lambda_{1/n}$ of action less than $e^{1/n}L$ correspond
  via $\phi_{-1/n}$ to the Reeb orbits for $\lambda_0$ of action less
  than $L$, and the latter stay outside of the neighborhood $U_{1/n}$
  of $\mc{K}$.  Let $n$ be so large.

  By Proposition~\ref{prop:ls}(c), the flow of the Liouville vector
  field $V$ starting on $Y_{1/n}$ for times in the interval $[-1/n,0]$
  defines an embedding of $[-1/n,0]\times Y_{1/n}$ into $X_{1/n}$.
  Let $X_{1/n}^0$ denote the image of this embedding.  We identify
  $X_{1/n}^0$ with $[0,1/n]\times Y_{1/n}$ such that $Y_{1/n}$ is
  identified with $\{1/n\}\times Y_{1/n}$, and the Liouville vector
  field $V=\partial/\partial s$, where $s$ denotes the $[0,1/n]$
  coordinate.  Then $\{0\}\times Y_{1/n}$ defines a hypersurface in
  $X_{1/n}$ which includes $Y_0\setminus U_{1/n}\subset \partial
  X_{1/n}$, and which also passes into the interior of $X_{1/n}$.  Let
  $X_{1/n}^1$ denote $X_{1/n}\setminus X_{1/n}^0$.  We can now
  decompose the completed cobordism $\overline{X_{1/n}}$ as
\begin{equation}
\label{eqn:X1n}
\overline{X_{1/n}} = ((-\infty,0]\times Y_0) \cup X_{1/n}^1
\cup ([0,\infty)\times Y_{1/n}),
\end{equation}
where $X_{1/n}^0$ corresponds to $[0,1/n]\times Y_{1/n}$ in
\eqref{eqn:X1n}.

We now choose a cobordism-admissible almost complex structure $J$ on
$\overline{X_{1/n}}$ in four steps as follows.  First, let $J_+$ be an
almost complex structure on $\R\times Y_{1/n}$ which is
symplectization-admissible with respect to $\lambda_{1/n}$ and
ECH-generic.  Require $J$ to agree with $J_+$ on $[1/n,\infty)\times
Y_{1/n}$.  Second, extend $J$ over $[0,1/n]\times Y_{1/n}$ by setting
$J=J_+$ on $\Ker(\lambda_{1/n})$, and $J(\partial_s)=f(s)R_{1/n}$,
where $R_{1/n}$ denotes the Reeb vector field associated to
$\lambda_{1/n}$, and $f:[0,1/n]\to\R$ is a positive function which
equals $1$ near $s=1/n$ and which equals $e^{1/n}$ near $s=0$.  Third,
extend $J$ over $(-\infty,0]\times Y_0$ so that it agrees with an
almost complex structure $J_-$ on $\R\times Y_0$ which is
symplectization-admissible for $\lambda_0$ and ECH-generic.  Note that
one can arrange for $J_-$ to be ECH-generic without disturbing the
previous choices because $J_+$ is ECH-generic.  To complete the
construction of $J$, choose an arbitrary $\omega$-compatible extension
of $J$ over $X_{1/n}^1$.

With the above choices, $[0,1/n]\times
(Y_{1/n}\setminus\widetilde{U}_{1/n})$ is a product region in the
sense of Definition~\ref{def:PR}.  In particular, let $\Theta$ be an
ECH generator for $\lambda_{1/n}$ of action less than $e^{1/n}L$.
Since the Reeb orbits in $\Theta$ stay out of the region
$\widetilde{U}_{1/n}$, there is a union of covers of product cylinders
(see Definition~\ref{def:PC}) in $\overline{X_{1/n}}$ from $\Theta$ to
$\phi_{-1/n}(\Theta)$.

We claim that if $C$ is any other $J$-holomorphic curve in
$\overline{X_{1/n}}$ from the above $\Theta$ to an ECH generator
$\Theta'$ for $\lambda_0$, then
\begin{equation}
\label{eqn:actionClaim}
e^{-1/n}\mc{A}(\Theta)>\mc{A}(\Theta'),
\end{equation}
where $\mc{A}$ denotes the symplectic action.  To prove
\eqref{eqn:actionClaim}, observe that the Liouville form $\lambda$ on
$\overline{X_{1/n}}$ agrees with $e^s\lambda_0$ on $(-\infty,0]\times
Y_0$ and agrees with $e^{s-1/n}\lambda_{1/n}$ on $[0,\infty)\times
Y_{1/n}$ in \eqref{eqn:X1n}.  Using Stokes' theorem, we obtain
\[
\begin{split}
e^{-1/n}\mc{A}(\Theta) - \mc{A}(\Theta')  = & \int_{C\cap ([0,\infty)\times
  Y_{1/n})}d\left(e^{-1/n}\lambda_{1/n}\right)\\
&+ \int_{C\cap X_{1/n}^1}d\lambda\\
& + \int_{C\cap ((-\infty,0]\times Y_0)}d\lambda_0.
\end{split}
\]
The first and third integrals on the right are pointwise nonnegative,
and zero only where $C$ is tangent to $\partial_s$.  The second
integral on the right is pointwise positive.  We conclude that
$e^{-1/n}\mc{A}(\Theta)-\mc{A}(\Theta')\ge 0$, with equality if and
only if $C$ is a union of covers of product cylinders.

Consider now the chain map
$\hat{\Phi}^{e^{1/n}L}(X_{1/n},\lambda)$ inducing \eqref{eqn:cob1}
provided by the Holomorphic Curves axiom in Theorem~\ref{thm:cob}.  It
follows from \eqref{eqn:actionClaim} that this chain map is
a composition of chain maps
\[
ECC_*^{e^{1/n}L}(Y_{1/n},\lambda_{1/n};J_+)\longrightarrow
ECC_*^L(Y_0,\lambda_0;J_-) \longrightarrow ECC_*^{e^{1/n}L}(Y_0,\lambda_0;J_-),
\]
where the map on the right is the inclusion, and the map on the left
is triangular with respect to the identification of generators induced
by $\phi_{-1/n}$.  In particular the left map is an isomorphism of
chain complexes, and hence induces an isomorphism on homology.
\end{proof}

To proceed, we now define quantitative measures of the failure of the
cobordism map \eqref{eqn:surcob} to be an isomorphism.

\begin{definition}
\label{def:AB}
\begin{description}
\item{(a)}
Define $A$ to be the infimum of the set of real numbers $L$ such that 
the image of the inclusion-induced map
\[
\imath^L: H_*^L(Y_0,\lambda_0) \longrightarrow H_*(Y_0,\lambda_0)
\]
is not contained in the image of the cobordism map \eqref{eqn:surcob}.
\item{(b)}
Define $B$ to be the infimum of the set of real numbers $L$ such that
the kernel of the cobordism map
\[
\Phi^L(X,\lambda):
H_*^L(Y_1,\lambda_1) \longrightarrow H_*^L(Y_0,\lambda_0)
\]
is not contained in the kernel of the inclusion-induced map
\[
\imath^L:
H_*^L(Y_1,\lambda_1) \longrightarrow H_*(Y_1,\lambda_1).
\]
\end{description}
\end{definition}

\begin{lemma}
\label{lem:AB}
\begin{description}
\item{(a)} If \eqref{eqn:surcob} is not surjective, then $A<\infty$.
\item{(b)} If \eqref{eqn:surcob} is not injective, then $B<\infty$.
\end{description}
\end{lemma}

\begin{proof}
(a) If \eqref{eqn:surcob} is not surjective then there exists an
element of $H_*(Y_0,\lambda_0)$ which is not in the image; and by
\eqref{eqn:edr}, any given element of $H_*(Y_0,\lambda_0)$ comes from
$H_*^L(Y_0,\lambda_0)$ for some $L$.

(b) If \eqref{eqn:surcob} is not injective, then there exists a
nonzero element $H_*(Y_1,\lambda_1)$ which maps to zero in
$H_*(Y_0,\lambda_0)$.  We can represent the former by a chain $\zeta$
of action less than some $L$, and its image under the cobordism chain
map is the boundary of a chain with action less than some $L'$.  We may
assume that $L'\ge L$.  It then follows from the Inclusion axiom in
Theorem~\ref{thm:cob} that $\imath^{L,L'}[\zeta]$ is in the kernel of
the cobordism map $\Phi^{L'}(X,\lambda)$.  But $\imath^{L,L'}[\zeta]$
is not in the kernel of the inclusion-induced map $\imath^{L'}$,
because $\imath^{L'}\imath^{L,L'}[\zeta]=\imath^L[\zeta]\neq 0$ in
$H_*(Y_1,\lambda_1)$.  Thus $B\le L'$.
\end{proof}

In view of Lemma~\ref{lem:ni} and Definition~\ref{def:PhiX}, the chord
conjecture in the nondegenerate case now follows from:

\begin{theorem}
\label{thm:nondegenerate}
Let $Y_0$ be a closed oriented 3-manifold with a nondegenerate contact
form $\lambda_0$, and let $\mc{K}$ be a Legendrian knot in $(Y_0,\lambda_0)$.
\begin{description}
\item{(a)} If \eqref{eqn:surcob} is not surjective, then $\mc{K}$ has a
  Reeb chord of action $\le A$.
\item{(b)} If \eqref{eqn:surcob} is not injective, then $\mc{K}$ has a
  Reeb chord of action $\le B$.
\end{description}
\end{theorem}

\begin{proof}
  To prove part (a), it is enough to show that given $L>0$, if there
  is no Reeb chord of action $\le L$, then $A\ge L$.  (Because then if $A$
  is finite, then there exists a Reeb chord of action $\le A+1/n$ for every
  positive integer $n$, so a compactness argument shows that there
  exists a Reeb chord of action $\le A$.)  To show that $A\ge L$, it is
  enough to show that the image of the inclusion-induced map
\[
\imath^L: H_*^{L}(Y_0,\lambda_0) \longrightarrow H_*(Y_0,\lambda_0)
\]
is contained in the image of the cobordism map \eqref{eqn:surcob}.
By the Inclusion and Direct Limit axioms in Theorem~\ref{thm:cob} we
have a commutative diagram
\[
\begin{CD}
H_*^{e^{1/n}L}(Y_{1/n},\lambda_{1/n})
@>{\Phi^{e^{1/n}L}(X_{1/n},\lambda)}>> H_*^{e^{1/n}L}(Y_0,\lambda_0) \\
@VV{\imath^{e^{1/n}L}}V @VV{\imath^{e^{1/n}L}}V \\
H_*(Y_{1/n},\lambda_{1/n})
@>{\Phi(X_{1/n})}>> H_*(Y_0,\lambda_0)
\end{CD}
\]
If $n$ is sufficiently large as in Lemma~\ref{lem:main}, then it
follows that we have a commutative diagram
\[
\begin{CD}
  & & H_*^{e^{1/n}L}(Y_{1/n},\lambda_{1/n}) @>{\simeq}>> H_*^L(Y_0,\lambda_0) \\
  & & @VV{\imath^{e^{1/n}L}}V @VV{\imath^L}V \\
  H_*(Y_1,\lambda_1) @>{\simeq}>> H_*(Y_{1/n},\lambda_{1/n}) @>{\Phi(X_{1/n})}>>
  H_*(Y_0,\lambda_0).
\end{CD}
\]
Here the lower left arrow is induced by the cobordism $\overline{X
  \setminus X_{1/n}}$ from $Y_1$ to $Y_{1/n}$; this map is an
isomorphism because the cobordism $\overline{X\setminus X_{1/n}}$ is
diffeomorphic to the product $[0,1]\times Y_1$, and product cobordisms
induce isomorphisms on Seiberg-Witten Floer cohomology.  In addition,
the composition of the lower two arrows is the cobordism map
\eqref{eqn:surcob}, by the Composition axiom in Theorem~\ref{thm:cob}
(or by the composition property for $\widehat{HM}^*$).
The statement we need to prove now follows by chasing the diagram.

The proof of part (b) is similar to the proof of part (a). It is
enough to show that if there is no Reeb chord of action $\le L$, then the
kernel of the cobordism map $H_*^L(Y_1,\lambda_1)\to
H_*^L(Y_0,\lambda_0)$ is contained in the kernel of the
inclusion-induced map $H_*^L(Y_1,\lambda_1)\to H_*(Y_1,\lambda_1)$.
To do so, let $n$ be sufficiently large as in Lemma~\ref{lem:main}
(with $L$ replaced by $e^{-1/n}L$).  Then by the Inclusion and Direct
Limit axioms in Theorem~\ref{thm:cob}, we have a commutative diagram
\[
\begin{CD}
H_*^L(Y_1,\lambda_1) @>{\Phi^L(\overline{X\setminus X_{1/n}} ,
  \,\lambda)}>> H_*^L(Y_{1/n},\lambda_{1/n}) @>{\simeq}>>
H_*^{e^{-1/n}L}(Y_0,\lambda_0) \\
@VV{\imath^L}V @VV{\imath^L}V @VV{\imath^{e^{-1/n}L,L}}V\\
H_*(Y_1,\lambda_1) @>{\simeq}>> H_*(Y_{1/n},\lambda_{1/n}) & &
H_*^L(Y_0,\lambda_0),
\end{CD}
\]
where the composition of the two rightmost arrows is
$\Phi^L(X_{1/n},\lambda)$.  By the Composition axiom in
Theorem~\ref{thm:cob}, the composition of the three arrows from
$H_*^L(Y_1,\lambda_1)$ to $H_*^L(Y_0,\lambda_0)$ is
$\Phi^L(X,\lambda)$.  Now suppose $\zeta\in H_*^L(Y_1,\lambda_1)$ maps
to zero in $H_*^L(Y_0,\lambda_0)$.  Since the latter is the direct
limit of $H_*^{e^{-1/n}L}(Y_0,\lambda_0)$ as $n\to \infty$, it follows
that if $n$ is chosen sufficiently large then $\zeta$ maps to zero in
$H_*^{e^{-1/n}L}(Y_0,\lambda_0)$.  Chasing the diagram then shows that
$\zeta$ maps to zero in $H_*(Y_1,\lambda_1)$, as required.
\end{proof}

\section{The chord conjecture: degenerate case}
\label{sec:ccdegenerate}

We now use Theorem~\ref{thm:nondegenerate} for the nondegenerate case
to deduce the chord conjecture when $\lambda_0$ is degenerate.

When $\lambda_0$ is degenerate, one can repeat the surgery
construction from \S\ref{sec:ls}, to obtain an exact symplectic
cobordism $(X,\lambda)$ as before, now with degenerate contact forms
$\lambda_{1/n}$ on the hypersurfaces $Y_{1/n}$.  However by
\eqref{eqn:et} one can perturb these hypersurfaces slightly, as well
as the boundary hypersurfaces $Y_0$ and $Y_1$, in the completion
$\overline{X}$, so as to make the induced contact forms on them
nondegenerate.  In particular, for each positive integer $k$ we can
find functions $f_{1/n}^{(k)}$ on $Y_{1/n}$ and $f_0^{(k)}$ on $Y_0$
with
\[
\|f_{1/n}^{(k)}\|_{C^1},\; \|f_0^{(k)}\|_{C^1} <1/k
\]
such that we have an exact cobordism $X^{(k)}$ as in \S\ref{sec:ls},
contained in $\overline{X}$ with the same Liouville form $\lambda$
(which we henceforth omit from the notation), with nondegenerate
contact forms $\lambda_{1/n}^{(k)}=e^{f_{1/n}^{(k)}}\lambda_{1/n}$ on
$Y_{1/n}$ and $\lambda_0^{(k)}=e^{f_0^{(k)}}\lambda_0$ on $Y_0$.  We
can also assume that
\begin{equation}
\label{eqn:convenient}
f_0^{(1)}>f_0^{(k)}, \quad\quad f_1^{(1)}>f_1^{(k)},
\end{equation}
which will be convenient below.

Now let $A(k)$ and $B(k)$ denote the upper bounds on the action of a
Reeb chord coming from the cobordism $X^{(k)}$ from
$(Y_1,\lambda_{1}^{(k)})$ to $(Y_0,\lambda_{0}^{(k)})$.  By
Theorem~\ref{thm:nondegenerate} and a compactness argument for Reeb
chords explained at the end of this section, to prove the chord
conjecture for $\lambda_0$ it is enough to show that $A(k)$ and $B(k)$
stay bounded as $k\to\infty$.  In fact we have:

\begin{proposition}
\label{prop:continuity}
$A(k)\le A(1)$ and $B(k)\le B(1)$.
\end{proposition}

\begin{proof}
  By the first part of \eqref{eqn:convenient}, there is a subset
  $X_+^{(k)}$ of $\overline{X}$, diffeomorphic to $[0,1]\times Y_1$,
  which defines an exact symplectic cobordism from
  $(Y_1,\lambda_1^{(1)})$ to $(Y_1,\lambda_1^{(k)})$.  Likewise, by
  the second part of \eqref{eqn:convenient}, there is a subset
  $X_-^{(k)}$ of $\overline{X}$, diffeomorphic to $[0,1]\times Y_0$,
  which is an exact symplectic cobordism from $(Y_0,\lambda_0^{(1)})$
  to $(Y_0,\lambda_0^{(k)})$.  Let $X_0^{(k)}$ denote compact subset
  of $\overline{X}$ bounded by the negative boundary of $X_+^{(k)}$
  and the positive boundary of $X_-^{(k)}$.  This is an exact
  symplectic cobordism from $(Y_1,\lambda_1^{(k)})$ to
  $(Y_0,\lambda_0^{(1)})$, and we have the compositions
  $X_0^{(k)}\circ X_+^{(k)} = X^{(1)}$ and $X_-^{(k)}\circ
  X_0^{(k)} = X^{(k)}$.

Now fix $L$ and consider the diagram
\[
\begin{CD}
H_*^L(Y_1,\lambda_1^{(1)}) @>{\Phi^L(X_+^{(k)})}>>
H_*^L(Y_1,\lambda_1^{(k)}) @>{\Phi^L(X_0^{(k)})}>>
H_*^L(Y_0,\lambda_0^{(1)}) @>{\Phi^L(X_-^{(k)})}>>
H_*^L(Y_0,\lambda_0^{(k)}) \\
@VV{\imath^L}V
@VV{\imath^L}V
@VV{\imath^L}V
@VV{\imath^L}V
\\
H_*(Y_1,\lambda_1^{(1)}) @>{\Phi(X_+^{(k)})}>{\simeq}>
H_*(Y_1,\lambda_1^{(k)}) @>{\Phi(X_0^{(k)})}>>
H_*(Y_0,\lambda_0^{(1)}) @>{\Phi(X_-^{(k)})}>{\simeq}>
H_*(Y_0,\lambda_0^{(k)}).
\end{CD}
\]
This diagram commutes by the Inclusion axiom, and the maps
$\Phi(X_\pm^{(k)})$ are isomorphisms because the cobordisms
$X_\pm^{(k)}$ are diffeomorphic to products, which induce isomorphisms
on Seiberg-Witten Floer cohomology.  By the Composition axiom, the
composition of the two upper left horizontal arrows is
$\Phi^L(X^{(1)})$, the composition of the two upper right horizontal
arrows is $\Phi^L(X^{(k)})$, the composition of the two lower left
horizontal arrows is $\Phi(X^{(1)})$, and the composition of the two
lower right horizontal arrows is $\Phi(X^{(k)})$.

To prove that $A(k)\le A(1)$, it is enough to show that if the image of
$\imath^L:H_*^L(Y_0,\lambda_0^{(k)})\to H_*(Y_0,\lambda_0^{(k)})$ is
contained in the image of $\Phi(X^{(k)})$, then the image of
$\imath^L: H_*^L(Y_0,\lambda_0^{(1)})\to H_*(Y_0,\lambda_0^{(1)})$ is
contained in the image of $\Phi(X^{(1)})$.  This follows immediately
by chasing the above diagram.

To prove that $B(k)\le B(1)$, it is enough to show that if the kernel
of $\Phi^L(X^{(k)})$ is contained in the kernel of
$\imath^L:H_*^L(Y_1,\lambda_1^{(k)})\to H_*(Y_1,\lambda_1^{(k)})$,
then the kernel of $\Phi^L(X^{(1)})$ is contained in the kernel of
$\imath^L:H_*^L(Y_1,\lambda_1^{(1)})\to H_*(Y_1,\lambda_1^{(1)})$.
This also follows immediately from the above diagram.
\end{proof}

To conclude, we have:

\begin{proof}[Proof of Theorem~\ref{thm:cc}.]
Let $Y_0$ be a closed oriented 3-manifold with a contact form
$\lambda_0$, and let $\mc{K}$ be a Legendrian knot in
$(Y_0,\lambda_0)$.

If $\lambda_0$ is nondegenerate, then it follows
from Lemma~\ref{lem:ni} and Definition~\ref{def:PhiX} that the map
\eqref{eqn:surcob} is not an isomorphism.  Let $A,B\in[0,\infty]$
be the numbers in Definition~\ref{def:AB}. By Lemma~\ref{lem:AB} we have
$\min(A,B)<\infty$, and by Theorem~\ref{thm:nondegenerate} the knot $\mc{K}$
has a Reeb chord of length at most $\min(A,B)$.

If $\lambda_0$ is degenerate, let $\{\lambda_0^{(k)}\}_{k=1,2,\ldots}$
be a sequence of nondegenerate perturbations of $\lambda_0$ as
described at the beginning of this section, and let
$A(k),B(k)$ denote the corresponding quantities from
Definition~\ref{def:AB}.  By the nondegenerate case, $\mc{K}$ has a
Reeb chord $\gamma_k$ for $\lambda_0^{(k)}$ of length at most
$\min(A(k),B(k))<\infty$.  By Proposition~\ref{prop:continuity}, the
length of $\gamma_k$ has a $k$-independent upper bound.  Thus we can
pass to a subsequence such that the lengths of the Reeb chords
$\gamma_k$ converge to a real number $L$.  We can also pass to a
subsequence so that the starting and ending points of the Reeb chords
$\gamma_k$ converge to points $y_0,y_L\in\mc{K}$.  Now as
$k\to\infty$, the $1$-form $\lambda_0^{(k)}$ converges to $\lambda_0$
in $C^1$, and so the Reeb vector field for $\lambda_0^{(k)}$ converges
to the Reeb vector field for $\lambda_0$ in $C^0$.  Consequently there
is a Reeb chord for $\lambda_0$ of length $L$ from $y_0$ to $y_L$.
\end{proof}

\end{document}